\documentclass[11pt]{article}

\usepackage[margin=1in]{geometry}
\usepackage[utf8]{inputenc}
\usepackage{dsfont}
\usepackage{amsmath}
\usepackage{amsfonts}
\usepackage{mathtools}
\usepackage{hyperref}
\usepackage[capitalise]{cleveref}

\usepackage{authblk}

\usepackage{amsthm}
\newtheorem{theorem}{Theorem}
\newtheorem{lemma}{Lemma}
\newtheorem{corollary}[lemma]{Corollary}

\DeclareMathOperator{\p}{Pr}
\newcommand{\pr}[1]{\p\left[#1\right]}
\DeclareMathOperator{\E}{\mathbb{E}}
\newcommand{\expect}[1]{\E\!\left[#1\right]}

\newcommand{\N}{\mathbb{N}}

\newcommand{\given}{\,\middle|\,}
\newcommand*\diff{\mathop{}\!\mathrm{d}}
\newcommand{\ind}{\mathds{1}}

\newcommand{\myceil}[1]{\left \lceil #1 \right \rceil}

\newcommand{\arthur}[1]{}

\date{}
\title{Revisiting the Random Subset Sum Problem}

\author[1]{Arthur C. W. Da Cunha}
\author[1,2]{Francesco D'Amore}
\author[1]{Frédéric Giroire}
\author[1]{Hicham Lesfari}
\author[1]{Emanuele Natale}
\author[3]{Laurent Viennot}
\affil[1]{Université Côte d'Azur, Inria Sophia Antipolis, CNRS}
\affil[2]{Aalto University}
\affil[3]{Inria Paris, IRIF
\protect\\\tt{\{arthur.carvalho-walraven-da-cunha,frederic.giroire,\protect\\hicham.lesfari,emanuele.natale,laurent.viennot\}@inria.fr \protect\\ francesco.damore@aalto.fi}%
}

\begin{document}
    \maketitle
    \begin{abstract}
The average properties of the well-known \emph{Subset Sum Problem} can be studied by the means of its randomised version, where we are given a target value $z$, random variables $X_1, \ldots, X_n$, and an error parameter $\varepsilon > 0$, and we seek a subset of the $X_i$s whose sum approximates $z$ up to error~$\varepsilon$.
In this setup, it has been shown that, under mild assumptions on the distribution of the random variables, a sample of size $\mathcal{O}(\log(1/\varepsilon))$ suffices to obtain, with high probability, approximations for all values in~$[-1/2, 1/2]$.
Recently, this result has been rediscovered outside the algorithms community, enabling meaningful progress in other fields.
In this work we present an alternative proof for this theorem, with a more direct approach and resourcing to more elementary tools.
\end{abstract}     \section{Introduction}\label{sec:intro}

In the \emph{Subset Sum Problem (SSP)}, one is given as input a set of $n$ integers $X = \{x_1, x_2, \ldots, x_n\}$ and a target value $z$, and wishes to decide if there exists a subset of $X$ that sums to $z$.
That is, one is to reason about a subset $S \subseteq [n]$ such that
$
    \sum_{i \in S} x_i = z.
$
The special case where $z$ is half of the sum of $X$ is known as the \emph{Number Partition Problem (NPP)}.
The converse reduction is also rather immediate.\footnote{
    To find a subset of $X$ summing to $z$, one only needs to solve the NPP for the set $X \cup \{2z, \sum_{i \in [n]} x_i\}$.
    By doing so, one of the parts must consist of the element $\sum_{i \in [n]} x_i$ alongside the desired subset.
}

Be it in either of these forms, the SSP finds applications in a variety of fields, ranging from combinatorial number theory \cite{Sun03} to cryptography \cite{GemmellJ01, KateG11}.
In complexity theory, the SSP is a well-known NP-complete problem, being a common base for NP-completeness proofs.
In fact, the NPP version figures among Garey and Johnson’s six basic NP-hard problems \cite{GareyJ79}.
Under certain circumstances, the SSP can be challenging even for heuristics that perform well for many other NP-hard problems \cite{Johnson91, Ruml96}, and a variety of dedicated algorithms have been proposed to solve it \cite{helm2018subset, bringmann2021near, jin2018simple, jin2021fast, esser2019low}.
Nonetheless, it is not hard to solve it in polynomial time if we restrict the input integers to a fixed range \cite{bellman1966dynamic}.
It suffices to recursively list all achievable sums using the first $i$ integers: we start with $A_0 = \{0\}$ and compute $A_{i+1}$ as $A_i \cup \{a + x_{i+1} \mid a \in A_i\}$.
For integers in the range $[0, R]$, the search space has size $\mathcal{O}(nR)$.

Studying how the problem becomes hard as we consider larger ranges of integers (relative to $n$) requires a randomised version of the problem, the \emph{Random Subset Sum Problem (RSSP)}, where the input values are taken as independently and identically distributed random variables.
In this setup, the work \cite{BorgsCP01} proved that the problem experiences a phase transition in its average complexity as the range of integers increases.

The result we approach in this work comes from related studies on the typical properties of the problem.
In \cite{lueker1998} the author proves that, under fairly general conditions, the expected minimal distance between a subset sum and the target value is exponentially small.
More specifically, they show the following result.
\begin{theorem}[Lueker, 1998]\label{thm:lueker}
    Let $X_1, \dots, X_n$ be independent uniform random variables over $[-1, 1]$, and let $\varepsilon \in (0, 1/3)$.
    There exists a universal constant $C > 0$ such that, if $n \ge C\log(1/\varepsilon)$, then, with probability at least $1 - \varepsilon$, for all $z \in [-1, 1]$ there exists $S_z \subseteq [n]$ for which
    \begin{equation*}\label{eq:lueker-thm}
        \Bigl\lvert z - \sum_{i \in S_z} X_i \Bigr\rvert \le \varepsilon.
    \end{equation*}
\end{theorem}
That is, a rather small number (of the order of $\log\frac{1}{\varepsilon}$) of random variables suffices to have a high probability of approximating not only a single target $z$, but all values in an interval.

Even though \cref{thm:lueker} is stated and proved for uniform random variables over $[-1, 1]$, it is not hard to extend the result to a wide class of distributions.\footnote{Distributions whose probability density function $f$ satisfies $f(x) \ge b$ for all $x \in [-a, a]$, for some constants $a, b > 0$ (see Corollary 3.3 from \cite{lueker1998}).}
With this added generality, the theorem becomes a powerful tool for the analysis of random structures, and has recently proven to be particularly useful in the field of Machine Learning, taking part in a proof of the Strong Lottery Ticket Hypothesis \cite{PensiaRNVP20} and in subsequent related works \cite{daCunhaNV22, FischerB21, BurkholzLMG21}, and in Federated Learning \cite{WangDMWLLHMRD21}.

Generalisations of the RSSP have played important roles in the study of random Knapsack problems \cite{beierRandomKnapsackExpected2003,beierProbabilisticAnalysisKnapsack2004}, and to random binary integer programs \cite{borstIntegralityGapBinary2022,borst2022discrepancy}.
In particular, the works \cite{becchettiMultidimensionalRandomSubset2022}, \cite{borst2022discrepancy}, and \cite{borstIntegralityGapBinary2022} recently provided an extension of \cref{thm:lueker} to multiple dimensions.
As for the equivalent Random Number Partitioning Problem, \cite{chen2022} recently generalised \cite{BorgsCP01} and the integer version of the RSSP to non-binary integer coefficients.

The simplicity and ubiquity of the SSP has granted the related results a special didactic place.
Be it as a first example of NP-complete problem \cite{GareyJ79}, a path to science communication \cite{Hayes02}, or simply as a frame for the demonstration of advanced techniques \cite{Mertens01}, it has been a tool to make important, but sometimes complicated, ideas easier to communicate.

This work offers a substantially simpler alternative to the original proof of \cref{thm:lueker} by following a general framework introduced in the context of the analysis of Rumour Spreading algorithms \cite{doerrRandomizedRumorSpreading2017}.
Originally, the work \cite{lueker1998} approaches \cref{thm:lueker} by considering the random variable associated to the proportion of the values in the interval $[-1, 1]$ that can be approximated up to error $\varepsilon$ by the sum of some subset of the first $t$ variables, $X_1, \ldots, X_t$.

After restricting to some specific types of subsets, they proceed to evaluate the expected per-round growth of this proportion, conditioned on the outcomes of $X_1, \ldots, X_t$.
Their strategy is to analyse this expected increase by martingale theory, which only becomes possible after a non-linear transformation of the variables of interest.
Those operations hinder any intuition for the obtained martingale.
Nonetheless, a subsequent application of the Azuma-Hoeffding bound \cite{azuma1967weighted} followed by a case analysis leads to the result.

The argument presented here starts in the same direction as the original one, tracking the mass of values with suitable approximations as we reveal the values of the random variables $X_1, \ldots, X_n$ one by one.
However, we quickly diverge from \cite{lueker1998}, managing to obtain an estimation of the expected growth of this mass without discarding any subset-sum.
We eventually restrict the argument to some types of subsets, but we do so at a point where the need for such restriction is clear.

We proceed to directly analyse the estimation obtained, without any transformations.
Following \cite{doerrRandomizedRumorSpreading2017}, this estimation reveals two expected behaviours in expectation, which can be analysed in a similar way:
as we consider the first variables, the proportion of approximated values grows very fast;
then, after a certain point, the proportion of non-approximable values decreases very fast.

We remark that, while \cref{thm:lueker} crucially relies on tools from martingale theory such as Azuma-Hoeffding's inequality, which are not part of standard Computer Science curricula, our argument makes use of much more elementary results\footnote{Namely, the intermediate value theorem, Markov's inequality, and standard Hoeffding bounds.}
which should make it accessible enough for an undergraduate course on randomised algorithms.
     \section{Our argument}
In this section, we provide an alternative argument for proving \cref{thm:lueker}.
It takes shape much like the pseudo-polynomial algorithm we described in the introduction.
Leveraging the recursive nature of the problem, we construct a process which, at time $t$, describes the proportion of the interval $[-1,1]$ that can be approximated by some subset of the first $t$ variables.

We will show that with a suitable number of uniform variables (proportional to $\log(1/\varepsilon)$) a factor of $1 - \varepsilon/2$ of the values in $[-1, 1]$ can be approximated up to error $\varepsilon$.
This implies that any $z \in [-1,1]$ which cannot be approximated within error $\varepsilon$ is at most $\varepsilon$ away from a value that can.
Therefore it is possible to approximate $z$ up to error $2\varepsilon$.

\subsection{Preliminaries}
Let $X_1, \ldots, X_n$ be realisations of random variables as in \cref{thm:lueker}, and, without loss of generality, fix $\varepsilon > 0$.
We say a value $z \in \mathbb{R}$ is \emph{$\varepsilon$-approximated at time $t$} if and only if there exists $S \subseteq [t]$ such that
$
    \lvert z-\sum_{i\in S} X_i \rvert < \varepsilon.
$
For $0 \le t \le n$, let $f_t\colon \mathbb{R} \to \{0, 1\}$ be the indicator function for the event ``$z$ is $\varepsilon$-approximated at time $t$''.
Therefore, we have $f_0 = \ind_{(-\varepsilon,\varepsilon)},$ since only the interval $(-\varepsilon,\varepsilon)$ can be approximated by an empty set of values.
From there, we can exploit the recurrent nature of the problem: a value $z$ can be $\varepsilon$-approximated at time $t+1$ if and only if either $z$ or $z - X_{t + 1}$ could already be approximated at time $t$.
This implies that for all $z \in \mathbb{R}$ we have that
\begin{equation}\label{eq:recurrence-f}
    f_{t+1}(z) = f_t(z) + \left(1-f_t(z)\right) f_t(z-X_{t+1}).
\end{equation}

To keep track of the proportion of values in $[-1, 1]$ that can be $\varepsilon$-approximated at each step, we define, for each $0 \le t \le n$, the random variable
$$
    v_t = \frac{1}{2}\int_{-1}^1 f_t(z) \diff z.
$$
For better readability, throughout the text we will refer to $v_t$ simply as ``the volume.''

As we mentioned, it suffices to show that, with high probability, at time $n$, enough of the interval is $\varepsilon$-approximated (more precisely, that $v_n \ge 1 - \varepsilon/2$) to conclude that the entire interval is $2\varepsilon$-approximated.

\subsubsection{Expected behaviour}

Our first lemma provides a lower bound on the expected value of $v_t$.

\begin{lemma}\label{lem:expectation-volume}
    For all $0 \le t < n$, it holds that
    $$
        \expect{v_{t+1} \given X_1,\dots,X_t} \ge v_t \left[ 1 + \frac{1}{4} \left(1-v_t\right)\right].
    $$
\end{lemma}
\begin{proof}
The definition of $v_t$ and the recurrence in \cref{eq:recurrence-f} give us that
\begin{align*}
    \expect{v_{t+1} \given X_1, \dots, X_t}
    &= \expect{\frac{1}{2}\int_{-1}^1 f_{t+1}(z) \diff z \,\bigg|\, X_1, \dots, X_t}\\
    &= \int_{-1}^1 \frac{1}{2} \left(\frac{1}{2} \int_{-1}^1 f_t(z) + \left(1-f_t(z)\right) f_t(z - x) \diff z\right) \diff x\\
    &= \frac{1}{2} \int_{-1}^1 f_t(z) \diff z \int_{-1}^1 \frac{1}{2} \diff x + \frac{1}{2} \int_{-1}^1 \frac{1}{2} \int_{-1}^1 \left(1-f_t(z)\right) f_t(z - x) \diff z \diff x\\
    &= v_t + \frac{1}{4} \int_{-1}^1 \left(1 - f_t(z)\right) \int_{-1}^1 f_{t}(z-x) \diff x \diff z\\
    &= v_t + \frac{1}{4} \int_{-1}^1 \left(1 - f_t(z)\right) \int_{z - 1}^{z + 1} f_t(y)\diff y \diff z,
\end{align*}
where the last equality holds by substituting $y = z - x$.
For the previous ones we apply basic properties of integrals and Fubini's theorem to change the order of integration.

We now look for a lower bound for the last integral in terms of $v_t$.
To this end, we exploit that, since all integrands are non-negative, for all $u \in [-1/2, 1/2]$ we have that
\begin{align*}
    \int_{-1}^1 \left(1 - f_t(z)\right) \int_{z - 1}^{z + 1} f_t(y)\diff y \diff z
    &\ge \int_{u - \frac{1}{2}}^{u + \frac{1}{2}} \left(1 - f_t(z)\right) \int_{z - 1}^{z + 1} f_t(y)\diff y \diff z\\
    &\ge \int_{u - \frac{1}{2}}^{u + \frac{1}{2}} \left(1 - f_t(z)\right) \int_{u - \frac{1}{2}}^{u + \frac{1}{2}} f_t(y)\diff y \diff z.
\end{align*}
Both inequalities come from range restrictions: in the first we use that $u \in [-1/2, 1/2]$ implies $[u - 1/2, u + 1/2] \subseteq [-1, 1]$; for the second, we have that $[u - 1/2, u + 1/2] \subseteq [z-1, z+1]$ for all $z \in [u - 1/2, u + 1/2]$.\arthur{say this before showing the integrals... people are scared of them}
\arthur{Split into two inequalities.}

To relate the expression to $v_t$ explicitly, we choose $u$ in a way that the window $[u - 1/2, u + 1/2]$ entails exactly half of $v_t$.
The existence of such $u$ may become clear by recalling the definition of $v_t$.
To make it formal, consider the function given by
$$
    h(u) = \frac{1}{2} \int_{u - \frac{1}{2}}^{u + \frac{1}{2}} f_t(y) \diff y,
$$
and observe that
$$
    \min\,\{h(-1/2),\, h(1/2)\} \le \frac{v_t}{2},
    \qquad\text{ and }\qquad
    \max\,\{h(-1/2),\, h(1/2)\} \ge \frac{v_t}{2}.
$$
Thus, by the intermediate value theorem, there exists $u^* \in [-1/2, 1/2]$ for which $h(u^*) = v_t/2$, that is, for which
$$
    \frac{1}{2} \int_{u^* - \frac{1}{2}}^{u^* + \frac{1}{2}} f_t(y) \diff y = \frac{v_t}{2}.
$$

Altogether, we can conclude that
\begin{align*}
    \expect{v_{t+1} \given X_1, \dots, X_t}
    &= v_t + \frac{1}{4} \int_{-1}^1 \left(1 - f_t(z)\right) \int_{z - 1}^{z + 1} f_t(y)\diff y \diff z\\
    &\ge v_t + \frac{1}{2} \int_{u^* - \frac{1}{2}}^{u^* + \frac{1}{2}} \left(1 - f_t(z)\right) \biggl(\frac{1}{2} \int_{u^* - \frac{1}{2}}^{u^* + \frac{1}{2}} f_t(y)\diff y\biggr) \diff z\\
    &= v_t + \left(\frac{1}{2} - \frac{v_t}{2}\right) \frac{v_t}{2}\\
    &= v_t \left[ 1 + \frac{1}{4} \left(1-v_t\right)\right].
\end{align*}
\end{proof}

\cref{lem:expectation-volume} tells us that, if $v_t$ were to behave as expected, it should grow exponentially up to $1/2$, at which point $1 - v_t$ starts to decrease exponentially.
The rest of the proof follows accordingly, with \cref{subseq:to_one_half} analysing the progress of $v_t$ up to one half, and \cref{subseq:from_one_half} analogously following the complementary value, $1 - v_t$, starting from one half.
By building on the results from \cref{subseq:to_one_half}, we obtain fairly straightforward proofs in \cref{subseq:from_one_half}.
Thus, the following subsection comprises the core of our argument.

\subsection{Growth of the volume up to 1/2}\label{subseq:to_one_half}
Arguably, the main challenge in analysing the RSSP is the existence of over-time dependencies and deciding how to overcome it sets much of the course the proof will take.
Our strategy consists in constructing another process which dominates the original one while being free of dependencies.

Let $\tau_{1}$ be the first time at which the volume exceeds $1/2$, that is, let
$$
    \tau_{1} = \min\{t \ge 0 : v_t > 1/2\}.
$$
We just proved that up to time $\tau_1$ the process $v_t$ enjoys exponential growth in expectation.
In the following lemma we apply a basic concentration inequality to translate this property into a constant probability of exponential growth for $v_t$ itself.
\begin{lemma}\label{lem:exit-proba}
    Given $\beta \in (0, 1/8)$, let $p_{\beta} = 1 - \frac{7}{8(1-\beta)}$.
    For all integers $0 \le t < \tau_1$ it holds that
    $$
        \pr{v_{t+1} \ge v_t(1+\beta) \given X_1, \dots, X_t,\, t < \tau_{1}} \geq p_{\beta}.
    $$
\end{lemma}
\begin{proof}
The result shall follow easily from reverse Markov's inequality \cite[Lemma 4]{boyd2006randomized} and the bound from \cref{lem:expectation-volume}.
However, doing so requires a suitable upper bound on $v_{t+1}$ and, while~$2v_t$ would serve the purpose, such bound does not hold in general.

We overcome this limitation by fixing $t$ and considering how much $v_t$ would grow in the next step if we were to consider only values $\varepsilon$-approximated at time $t$ that happen to lie in~$[-1, 1]$ after being translated by $X_{t+1}$.
Making it precise by the means of the recurrence in \cref{eq:recurrence-f}, we define
$$
    \tilde{v} = \frac{1}{2} \int_{-1}^1 \left[f_t(z) + \left(1-f_t(z)\right) f_t(z - X_{t+1}) \cdot \ind_{[-1,1]}(z - X_{t+1})\right] \diff z.
$$

This expression differs from the one for $v_{t+1}$ only by the inclusion of the characteristic function of $[-1, 1]$.
This not only implies that $\tilde{v} \le v_{t+1}$, but also that $\tilde{v}$ can replace $v_{t+1}$ in the bound from \cref{lem:expectation-volume}, since the argument provided there eventually restricts itself to integrals within $[-1, 1]$, trivialising $\ind_{[-1,1]}$.
Moreover, as we obtain $\tilde{v}$ without the influence of values from outside $[-1, 1]$, we must have~$\tilde{v} \le 2v_t$.
Finally, using that $t < \tau_1$ implies~$v_t < 1/2$ and chaining the previous conclusions in respective order\arthur{This region is quite bad.
We need to split the following align, but haven't managed to}, we conclude that
\begin{align*}
    \pr{v_{t+1} \ge v_t(1+\beta) \given X_1, \dots, X_t, t < \tau_{1}}
    &\ge \pr{\tilde{v} \ge v_t(1+\beta) \given X_1, \dots, X_t, t < \tau_{1}}\\
    &\ge \frac{\expect{\tilde{v} \given X_1,\dots,X_t, t < \tau_{1}} - v_t(1+\beta)}{2v_t - v_t(1+\beta)} \\
    &\ge \frac{ \frac{9}{8} v_t -v_t(1+\beta)}{2v_t-v_t(1+\beta)}\\
    &= 1 - \frac{7}{8(1-\beta)},
\end{align*}
where we applied the reverse Markov's inequality in the second step.
\end{proof}

The previous lemma naturally leads us to look for bounds on $\tau_1$, that is, to estimate the time needed for the process to reach volume $1/2$.
As expected, the exponential nature of the process yields a logarithmic bound.

\begin{lemma}\label{lem:tau1-bound}
    Let $t$ be an integer and given $\beta \in (0, 1/8)$, let $p_{\beta} = 1 - \frac{7}{8(1-\beta)}$ and
    $
        i^* = \left\lceil \frac{\log\frac{1}{2\varepsilon}}{\log(1+\beta)}\right\rceil.
    $
    If~$t \ge i^*/p_{\beta}$, then
    $$
        \pr{\tau_{1} \leq t} \ge 1 - \exp\left[-\frac{2p_{\beta}^2}{t}\left(t - \frac{i^*}{p_{\beta}}\right)^2\right].
    $$
\end{lemma}
\begin{proof}
The main idea behind the proof is to define a new random variable which stochastically dominates $\tau_1$ while being simpler to analyse.
We begin by discretising the domain $(0, 1/2]$ of the volume into sub-intervals $\{I_i\}_{0 \leq i \leq i^*}$ defined as follows:
\[
    \left\{
        \begin{aligned}
            &I_0 = (0, \varepsilon],\\
            &I_i = \left(\varepsilon (1 + \beta)^{i-1},\, \varepsilon (1 + \beta)^i\right] \text{ for } 1 \leq i < i^*,\\
            &I_{i^*} = \left(\varepsilon (1 + \beta)^{i^*-1},\, \frac{1}{2}\right],
        \end{aligned}
    \right.
\]
where $i^*$ is the smallest integer for which $\varepsilon \left(1+\beta\right)^{i^*} \ge 1/2$, that is,
$
    i^* = \left\lceil \frac{\log\frac{1}{2\varepsilon}}{\log(1+\beta)}\right\rceil.
$

Now, for each $i \geq 0$, we direct our interest to the number of steps required for $v_t$ to exit the sub-interval $I_i$ after first entering it.
By \cref{lem:exit-proba}, this number is majorised by a geometric random variable $Y_i \sim \textrm{Geom}(p_{\beta})$.
Therefore, we can conclude that $\tau_{1}$ is stochastically dominated by the sum of such variables, that is, for $t \in \N$, we have that
\begin{align}\label{eq:dominance-geo}
    \pr{\tau_{1} \ge t} \le \pr{\sum_{i=1}^{i^*} Y_i \ge t}.
\end{align}

Let $B_t \sim \textrm{Bin}(t,p_{\beta})$ be a binomial random variable.
For the sum of geometric random variables, it holds that
$
    \pr{\sum_{i=1}^{i^*} Y_i \le t} = \pr{B_t \ge i^*}.
$
Since $\expect{B_t} = t p_\beta$, the Hoeffding bound for binomial random variables \cite[Theorem 1.1]{dubhashi2011} implies that, for all $\lambda \ge 0$, we have that
$
    \Pr[B_t \le tp_{\beta} - \lambda] \le \exp(-2\lambda^2/t).
$
Setting $t$ such that $t p_{\beta} - \lambda = i^*$, we obtain that
\begin{align*}
    \pr{\sum_{i=1}^{i^*} Y_i \ge t}
    \le \pr{B_t \le i^*}
    \le \exp\left[-\frac{2}{t} \left(tp_{\beta} - i^*\right)^2\right]
    = \exp\left[-\frac{2p_{\beta}^2}{t} \left(t - \frac{i^*}{p_{\beta}}\right)^2\right],
\end{align*}
which holds as long as $\lambda = tp_\beta - i^* \ge 0$, that is, for all
$
    t \ge \frac{1}{p_{\beta}}\myceil{\frac{\log \frac{1}{2\varepsilon}}{\log (1 + \beta)}}.
$

The thesis follows by applying this to \cref{eq:dominance-geo} and passing to complementary events.
\end{proof}

\subsection{Growth of the volume from 1/2}\label{subseq:from_one_half}
Here we study the second half of the process: from the moment the volume reaches 1/2 up to the time it gets to $1 - \varepsilon/2$.
We do so by analysing the complementary stochastic process, i.e., by tracking, from time $\tau_1$ onwards, the proportion of the interval $[-1,1]$ that does not admit an $\varepsilon$-approximation.
More precisely, we consider the process $\{w_t\}_{t\ge 0}$, defined by $w_t = 1-v_{\tau_{1}+t}$.

We shall obtain results for $w_t$ similar to those we have proved for $v_t$.
Fortunately, building on the previous results makes those proofs quite straightforward.
We start by noting that a statement analogous to \cref{lem:expectation-volume} follows immediately from the definition of $w_{t+1}$ and \cref{lem:expectation-volume}.

\begin{corollary}\label{lem:expectation-volume-w}
    For all $t \ge 0$, it holds that
    $$
        \expect{w_{t+1} \given X_{1},\dots,X_{\tau_{1}+t}} \leq w_t \left[1 - \frac{1}{4} \left(1-w_t\right)\right].
    $$
\end{corollary}

Let $\tau_2$ the first time that $w_t$ gets smaller than or equal to $\varepsilon/2$, that is, let
$$
    \tau_2 = \min\left\{t \ge 0: w_t \le \varepsilon/2\right\}.
$$
The following lemma bounds this quantity, in analogy to \cref{lem:tau1-bound}.

\begin{lemma}\label{lem:tau2-bound}
    For all $t > 0$, it holds that
    $$
        \pr{\tau_2 \leq t} \geq 1 - \frac{1}{\varepsilon} \left(\frac{7}{8}\right)^t.
    $$
\end{lemma}
\begin{proof}
Applying that $1 - w_t = v_{\tau_1 + t} > 1/2$ to \cref{lem:expectation-volume-w} gives the bound
\begin{equation}\label{eq:UB-w}
    \expect{w_{t+1} \given X_{1},\dots,X_{\tau_{1}+t}}
    \le \frac{7}{8} w_t.
\end{equation}
Moreover, from the conditional expectation theory, for any two random variables $X$ and $Y$, we have $\expect{\expect{X \given Y}} = \expect{X}$.
From this and \cref{eq:UB-w}, we can conclude that
\begin{align*}
    \expect{w_t}
    = \expect{\expect{w_t \given X_{1},\dots,X_{\tau_1 + t - 1}}}
    \leq \frac{7}{8} \expect{w_{t-1}},
\end{align*}
which, by recursion, yields that
\begin{align*}
    \expect{w_t}
    \leq \left(\frac{7}{8}\right)^{t} \expect{w_0}
    \leq \frac{1}{2} \left(\frac{7}{8}\right)^{t}.
\end{align*}

Finally, by Markov's inequality,
\begin{align*}
    \pr{\tau_2 \ge t}
    \le \pr{w_t \ge \frac{\varepsilon}{2}}
    \le \frac{2\expect{w_t}}{\varepsilon}
    \le \frac{1}{\varepsilon} \left(\frac{7}{8}\right)^{t},
\end{align*}
and the thesis follows from considering the complementary event.
\end{proof}

\subsection{Putting everything together}
In this section we conclude our argument, finally proving \cref{thm:lueker}.
We first prove a more general statement and then detail how it implies the theorem.

Let $\tau = \tau_1 + \tau_2$, the first time at which the process $\{v_t\}_{t \ge 0}$ reaches at least $1-\varepsilon/2$.
\begin{lemma}\label{lem:final-tau}
Let $\varepsilon \in (0, 1/3)$.
There exist constants $C' > 0$ and $\kappa > 0$ such that for every $t \ge C' \log\frac{1}{\varepsilon}$, it holds that
$$
    \pr{\tau \le t} \geq 1 - 2\exp\left[-\frac{1}{\kappa t} \left(t - C'\log\frac{1}{\varepsilon}\right)^2\right].
$$
\end{lemma}
\begin{proof}
Let $\beta = \frac{1}{16}$ and $p_{\beta} = 1 - \frac{7}{8 (1 - \beta)} = \frac{1}{15}$.
The definition of $\tau$ allows us to apply \cref{lem:tau1-bound,lem:tau2-bound} quite directly.
Indeed if, for the sake of \cref{lem:tau1-bound}, we assume $t \ge \frac{2}{p_{\beta}}\bigl\lceil\frac{\log \frac{1}{2\varepsilon}}{\log (1 + \beta)}\bigr\rceil$, we have that
\begin{align}
    \pr{\tau \le t}
    &= \pr{\tau_1 + \tau_2 \le t} \nonumber\\
    &\ge \pr{\tau_1 \le t/2, \tau_2 \le t/2} \nonumber\\
    &\ge \pr{\tau_1 \le t/2} + \pr{\tau_2 \le t/2} - 1 \nonumber\\
    &\ge 1 - \exp\left[-\frac{p_{\beta}^2}{t}\left(t - \frac{2}{p_{\beta}}\myceil{\frac{\log \frac{1}{2\varepsilon}}{\log (1 + \beta)}}\right)^2\right] - \frac{1}{\varepsilon} \left(\frac{7}{8}\right)^{t/2} \nonumber\\
    &= 1 - \exp\left[-\frac{1}{15^2 t}\left(t - 30\myceil{\frac{\log \frac{1}{2\varepsilon}}{\log\frac{17}{16}}}\right)^2\right] - \frac{1}{\varepsilon} \left(\frac{7}{8}\right)^{t/2},\label{eq:trick}
\end{align}
where the second inequality holds by the union bound.
The remaining of the proof consists in computations to connect this expression to the one in the statement.

Consider the first exponential term in \cref{eq:trick}.
Taking $t \ge \frac{60}{\log \frac{17}{16}} \cdot \log\frac{1}{\varepsilon}$, since $\varepsilon < 1/3$, it follows that
\begin{align*}
    \exp\left[-\frac{1}{15^2 t}\left(t - 30\myceil{\frac{\log \frac{1}{2\varepsilon}}{\log\frac{17}{16}}}\right)^2\right]
    &\le \exp\left[-\frac{1}{15^2 t}\left(t - \frac{60}{\log\frac{17}{16}} \cdot \log\frac{1}{\varepsilon}\right)^2\right].
\end{align*}

Now, consider the second exponential term in \cref{eq:trick}.
It holds that
\begin{align*}
    \frac{1}{\varepsilon} \left(\frac{7}{8}\right)^{\frac{t}{2}}
    &= \exp\left[\log\frac{1}{\varepsilon} - \frac{t}{2} \log\frac{8}{7}\right] \\
    &\le \exp\left[\log\frac{1}{\varepsilon} - \frac{t}{15}\right]
    = \exp\left[-\frac{1}{15} \cdot \frac{1}{t - 15 \cdot \log\frac{1}{\varepsilon}} \cdot \left(t - 15 \cdot \log\frac{1}{\varepsilon}\right)^2\right].
\end{align*}
Moreover, for $t \ge 15 \cdot \log\frac{1}{\varepsilon}$,
\begin{align*}
    \exp\left[-\frac{1}{15} \cdot \frac{1}{t - 15 \cdot \log\frac{1}{\varepsilon}} \cdot \left(t - 15 \cdot \log\frac{1}{\varepsilon}\right)^2\right]
    &\le \exp\left[- \frac{1}{15t} \left(t - 15 \cdot \log\frac{1}{\varepsilon}\right)^2 \right]
    \\&\le \exp\left[-\frac{1}{15^2 t}\left(t - \frac{60}{\log\frac{17}{16}} \cdot \log\frac{1}{\varepsilon}\right)^2\right].
\end{align*}

Altogether, we have that
\begin{align*}
    \exp\left[-\frac{p_{\beta}^2}{t}\left(t - \frac{2}{p_{\beta}}\myceil{\frac{\log \frac{1}{2\varepsilon}}{\log (1 + \beta)}}\right)^2\right]
    + \frac{1}{\varepsilon}\cdot \left(\frac{7}{8}\right)^{t/2}
    &\le 2\exp\left[-\frac{1}{15^2 t}\left(t - \frac{60}{\log \frac{17}{16}} \cdot \log\frac{1}{\varepsilon}\right)^2\right],
\end{align*}
and the thesis follows by setting $\kappa = 15^2$ and $C' = 60/\log(17/16)$.
\end{proof}

The expression in the claim of \cref{lem:final-tau} can be reformulated as
$$
    \pr{v_t \ge 1 - \frac{\varepsilon}{2}} \ge 1 - 2\exp\left[-\frac{1}{\kappa t} \left(t - C'\log\frac{1}{\varepsilon}\right)^2\right];
$$
hence,
\cref{thm:lueker} follows by taking $C \ge 3C'$ and observing that once we can approximate all but an $\varepsilon/2$ proportion of the interval $[-1,1]$, any $z \in [-1,1]$ either is $\varepsilon$-approximated itself, or is at most $\varepsilon$ away from a value that is, which implies that $z$ is $2\varepsilon$-approximated.
 \bibliographystyle{alpha}
    \bibliography{biblio.bib}

\newcommand{\etalchar}[1]{$^{#1}$}
\begin{thebibliography}{WDM{\etalchar{+}}21}

\bibitem[Azu67]{azuma1967weighted}
Kazuoki Azuma.
\newblock Weighted sums of certain dependent random variables.
\newblock {\em Tohoku Mathematical Journal, Second Series}, 19(3):357--367,
  1967.

\bibitem[BCP01]{BorgsCP01}
Christian Borgs, Jennifer~T. Chayes, and Boris~G. Pittel.
\newblock Phase transition and finite-size scaling for the integer partitioning
  problem.
\newblock {\em Random Struct. Algorithms}, 19(3-4):247--288, 2001.

\bibitem[BdCC{\etalchar{+}}22]{becchettiMultidimensionalRandomSubset2022}
Luca Becchetti, Arthur Carvalho~Walraven da~Cunha, Andrea Clementi, Francesco
  d'~Amore, Hicham Lesfari, Emanuele Natale, and Luca Trevisan.
\newblock On the {{Multidimensional Random Subset Sum Problem}}.
\newblock report, {Inria \& Université Cote d'Azur, CNRS, I3S, Sophia
  Antipolis, France ; Sapienza Università di Roma, Rome, Italy ; Università
  Bocconi, Milan, Italy ; Università di Roma Tor Vergata, Rome, Italy}, 2022.

\bibitem[BDHK22]{borst2022discrepancy}
Sander Borst, Daniel Dadush, Sophie Huiberts, and Danish Kashaev.
\newblock A nearly optimal randomized algorithm for explorable heap selection.
\newblock {\em CoRR}, abs/2210.05982, 2022.

\bibitem[BDHT22]{borstIntegralityGapBinary2022}
Sander Borst, Daniel Dadush, Sophie Huiberts, and Samarth Tiwari.
\newblock On the integrality gap of binary integer programs with {{Gaussian}}
  data.
\newblock {\em Mathematical Programming}, 2022.

\bibitem[Bel66]{bellman1966dynamic}
Richard Bellman.
\newblock Dynamic programming.
\newblock {\em Science}, 153(3731):34--37, 1966.

\bibitem[BGPS06]{boyd2006randomized}
Stephen Boyd, Arpita Ghosh, Balaji Prabhakar, and Devavrat Shah.
\newblock Randomized gossip algorithms.
\newblock {\em IEEE transactions on information theory}, 52(6):2508--2530,
  2006.

\bibitem[BLMG22]{BurkholzLMG21}
Rebekka Burkholz, Nilanjana Laha, Rajarshi Mukherjee, and Alkis Gotovos.
\newblock On the existence of universal lottery tickets.
\newblock In {\em International Conference on Learning Representations}, 2022.

\bibitem[BV03]{beierRandomKnapsackExpected2003}
Rene Beier and Berthold Vöcking.
\newblock Random knapsack in expected polynomial time.
\newblock In {\em Proceedings of the Thirty-Fifth Annual {{ACM}} Symposium on
  {{Theory}} of Computing}, {{STOC}} '03, pages 232--241. {Association for
  Computing Machinery}, 2003.

\bibitem[BV04]{beierProbabilisticAnalysisKnapsack2004}
Rene Beier and Berthold Vöcking.
\newblock Probabilistic analysis of knapsack core algorithms.
\newblock In {\em Proceedings of the Fifteenth Annual {{ACM-SIAM}} Symposium on
  {{Discrete}} Algorithms}, {{SODA}} '04, pages 468--477. {Society for
  Industrial and Applied Mathematics}, 2004.

\bibitem[BW21]{bringmann2021near}
Karl Bringmann and Philip Wellnitz.
\newblock On near-linear-time algorithms for dense subset sum.
\newblock In {\em Proceedings of the 2021 ACM-SIAM Symposium on Discrete
  Algorithms (SODA)}, pages 1777--1796. SIAM, 2021.

\bibitem[CJRS22]{chen2022}
Xi~Chen, Yaonan Jin, Tim Randolph, and Rocco~A. Servedio.
\newblock Average-case subset balancing problems.
\newblock In Joseph~(Seffi) Naor and Niv Buchbinder, editors, {\em Proceedings
  of the 2022 {ACM-SIAM} Symposium on Discrete Algorithms, {SODA} 2022, Virtual
  Conference / Alexandria, VA, USA, January 9 - 12, 2022}, pages 743--778.
  {SIAM}, 2022.

\bibitem[dCNV22]{daCunhaNV22}
Arthur da~Cunha, Emanuele Natale, and Laurent Viennot.
\newblock Proving the strong lottery ticket hypothesis for convolutional neural
  networks.
\newblock In {\em International Conference on Learning Representations}, 2022.

\bibitem[DK17]{doerrRandomizedRumorSpreading2017}
Benjamin Doerr and Anatolii Kostrygin.
\newblock Randomized rumor spreading revisited.
\newblock In Ioannis Chatzigiannakis, Piotr Indyk, Fabian Kuhn, and Anca
  Muscholl, editors, {\em 44th International Colloquium on Automata, Languages,
  and Programming, {ICALP} 2017, July 10-14, 2017, Warsaw, Poland}, volume~80
  of {\em LIPIcs}, pages 138:1--138:14. Schloss Dagstuhl - Leibniz-Zentrum
  f{\"{u}}r Informatik, 2017.

\bibitem[DP09]{dubhashi2011}
Devdatt~P. Dubhashi and Alessandro Panconesi.
\newblock {\em Concentration of Measure for the Analysis of Randomized
  Algorithms}.
\newblock Cambridge University Press, 2009.

\bibitem[EM19]{esser2019low}
Andre Esser and Alexander May.
\newblock Low weight discrete logarithms and subset sum in $2^{0.65 n}$ with
  polynomial memory.
\newblock {\em Cryptology ePrint Archive}, 2019.

\bibitem[FB21]{FischerB21}
Jonas Fischer and Rebekka Burkholz.
\newblock Towards strong pruning for lottery tickets with non-zero biases.
\newblock {\em CoRR}, abs/2110.11150, 2021.

\bibitem[GJ79]{GareyJ79}
M.~R. Garey and David~S. Johnson.
\newblock {\em Computers and Intractability: {A} Guide to the Theory of
  NP-Completeness}.
\newblock W. H. Freeman, 1979.

\bibitem[GJ01]{GemmellJ01}
Peter Gemmell and Anna~M. Johnston.
\newblock Analysis of a subset sum randomizer.
\newblock {\em {IACR} Cryptol. ePrint Arch.}, page~18, 2001.

\bibitem[Hay02]{Hayes02}
Brian Hayes.
\newblock The easiest hard problem.
\newblock {\em American Scientist}, 90:113--117, 2002.

\bibitem[HM18]{helm2018subset}
Alexander Helm and Alexander May.
\newblock Subset sum quantumly in 1.17\^{} n.
\newblock In {\em 13th Conference on the Theory of Quantum Computation,
  Communication and Cryptography (TQC 2018)}. Schloss Dagstuhl-Leibniz-Zentrum
  fuer Informatik, 2018.

\bibitem[JAMS91]{Johnson91}
David~S. Johnson, Cecilia~R. Aragon, Lyle~A. McGeoch, and Catherine Schevon.
\newblock Optimization by simulated annealing: an experimental evaluation; part
  ii, graph coloring and number partitioning.
\newblock {\em Operations research}, 39(3):378--406, 1991.

\bibitem[JVW21]{jin2021fast}
Ce~Jin, Nikhil Vyas, and Ryan Williams.
\newblock Fast low-space algorithms for subset sum.
\newblock In {\em Proceedings of the 2021 ACM-SIAM Symposium on Discrete
  Algorithms (SODA)}, pages 1757--1776. SIAM, 2021.

\bibitem[JW18]{jin2018simple}
Ce~Jin and Hongxun Wu.
\newblock A simple near-linear pseudopolynomial time randomized algorithm for
  subset sum.
\newblock {\em arXiv preprint arXiv:1807.11597}, 2018.

\bibitem[KG11]{KateG11}
Aniket Kate and Ian Goldberg.
\newblock Generalizing cryptosystems based on the subset sum problem.
\newblock {\em Int. J. Inf. Sec.}, 10(3):189--199, 2011.

\bibitem[Lue98]{lueker1998}
George~S. Lueker.
\newblock Exponentially small bounds on the expected optimum of the partition
  and subset sum problems.
\newblock {\em Random Structures and Algorithms}, 12:51--62, 1998.

\bibitem[Mer01]{Mertens01}
Stephan Mertens.
\newblock A physicist's approach to number partitioning.
\newblock {\em Theor. Comput. Sci.}, 265(1-2):79--108, 2001.

\bibitem[PRN{\etalchar{+}}20]{PensiaRNVP20}
Ankit Pensia, Shashank Rajput, Alliot Nagle, Harit Vishwakarma, and Dimitris~S.
  Papailiopoulos.
\newblock Optimal lottery tickets via subset sum: Logarithmic
  over-parameterization is sufficient.
\newblock In Hugo Larochelle, Marc'Aurelio Ranzato, Raia Hadsell,
  Maria{-}Florina Balcan, and Hsuan{-}Tien Lin, editors, {\em Advances in
  Neural Information Processing Systems 33: Annual Conference on Neural
  Information Processing Systems 2020, NeurIPS 2020, December 6-12, 2020,
  virtual}, 2020.

\bibitem[RNMS96]{Ruml96}
Wheeler Ruml, J.~Thomas Ngo, Joe Marks, and Stuart~M Shieber.
\newblock Easily searched encodings for number partitioning.
\newblock {\em Journal of Optimization Theory and Applications},
  89(2):251--291, 1996.

\bibitem[Sun03]{Sun03}
Zhi-Wei Sun.
\newblock Unification of zero-sum problems, subset sums and covers of z.
\newblock {\em Electronic Research Announcements of The American Mathematical
  Society}, 9:51--60, 2003.

\bibitem[WDM{\etalchar{+}}21]{WangDMWLLHMRD21}
Chenghong Wang, Jieren Deng, Xianrui Meng, Yijue Wang, Ji~Li, Sheng Lin, Shuo
  Han, Fei Miao, Sanguthevar Rajasekaran, and Caiwen Ding.
\newblock A secure and efficient federated learning framework for {NLP}.
\newblock In Marie{-}Francine Moens, Xuanjing Huang, Lucia Specia, and
  Scott~Wen{-}tau Yih, editors, {\em Proceedings of the 2021 Conference on
  Empirical Methods in Natural Language Processing, {EMNLP} 2021, Virtual Event
  / Punta Cana, Dominican Republic, 7-11 November, 2021}, pages 7676--7682.
  Association for Computational Linguistics, 2021.

\end{thebibliography}
\end{document}